\documentclass[reqno,english]{amsart}

\usepackage{packages}
\usepackage{macros}

\usepackage{natbib}

\usepackage{algorithm}
\usepackage{algpseudocode}

\newcommand*{\supp}{\mathrm{supp}}

\def\Exp{{\mathbb E}}
\def\Event{{\mathcal E}}

\def\Prob{{\mathbb P}}

\def\calF{{\mathcal{F}}}

\def\col{{\rm col}}
\def\row{{\rm row}}
\def\bfE{{\mathbf E}}
\def\bfC{{\mathbf C}}

\newtheorem*{theorem*}{Theorem}
\newtheorem*{question}{Question}
\newtheorem*{mmodel*}{Matrix model}

\def\online{\mathrm{ondisc}}
\def\disc{\mathrm{disc}}
\def\Cspread{C_{\text{\tiny{\ref{lemma:lb spread}}}}}
\def\excep{{\text{\tiny{\ref{lemma:exceptional row bound}}}}}
\newcommand*{\lb}{\text{\tiny{\ref{thm:lb}}}}
\newcommand*{\ub}{\text{\tiny{\ref{thm:ub}}}}
\newcommand*{\alg}{\text{\tiny{\ref{alg}}}}

\title {A threshold for online balancing of sparse i.i.d. vectors}

\author[D.\ J.\ Altschuler \and K.\ Tikhomirov]{
        Dylan J. Altschuler \and 
        Konstantin Tikhomirov 
        }

\address{Dylan J. Altschuler, Department of Mathematical Sciences, Carnegie Mellon University.}
\address{Konstantin Tikhomirov, Department of Mathematical Sciences, Carnegie Mellon University.}

\begin{document}

\begin{abstract}
Consider the task of \textit{online} vector balancing for stochastic arrivals $(X_i)_{i \in [T]}$, where the time horizon satisfies $T = \Theta(n)$, and the $X_i$ are i.i.d uniform $d$--sparse $n$--dimensional binary vectors,
with $2\leq d \le (\log\log n)^2/\log\log\log n$.
We show that for this range of parameters,
every online algorithm incurs discrepancy at least
$\Omega(\log \log n)$,
and there is an efficient algorithm which achieves a matching discrepancy bound of $O(\log\log n)$ w.h.p. This establishes an asymptotic gap, both existential and algorithmic, between the online and offline versions of the average--case Beck--Fiala problem. Strikingly, the optimal online discrepancy in the considered setting is order $\log \log n$, independent of $d$ and the norms of the vectors $(X_i)_i$. Our assumptions on $d$ are nearly optimal, as this independence ceases when $d=\omega((\log\log n)^2)$.
\end{abstract}

\maketitle
\section{Introduction}
Vector balancing is a fundamental task in combinatorics with numerous applications in algorithm design and optimization, ranging from rounding integer programs to experimental design \cite{beck-fiala,experimental}. A vector balancing problem is the task of assigning signs $\sigma_i \in \{-1,+1\}$ to vectors $X_i \in \RR^n$ with the goal of minimizing $\|\sum X_i\sigma_i\|_\infty$. Combinatorial discrepancy, a classical and widely studied quantity, corresponds to the case that all of the vectors $X_i$ are known beforehand:
\[
    \disc(X_1,\dots,X_T) = \min_{\sigma \in \{-1,1\}^T}\Big\|\sum_{i \in [T]} X_i \sigma_i\Big\|_\infty \,.
\]
Another classical setting, of particular interest in applications, is \textit{online} discrepancy minimization \cite{spencer-online}: the $X_i$ are revealed sequentially, and $\sigma_i$ must be chosen immediately and irrevocably upon revealing $X_i$. A natural motivation for this setting is that many of optimization problems connected to discrepancy minimization, such as bin packing and job scheduling \cite{bin}, have important online analogues. 

A rich body of works has developed estimates on the objective value for the online discrepancy minimization problem under various assumptions on the vectors $X_i$ (see, e.g., \cite{selfbalance,bansal-survey,bansal-online,optimal-online} and the references therein). Two core settings in this research direction are the so--called Koml\'os and Beck--Fiala settings, corresponding respectively to the cases that each $X_i$ is a unit vector or each $X_i$ is a $d$--sparse binary vector. Strikingly, in the Koml\'os setting\footnote{To the best of our knowledge, an analogous result does not seem to be recorded in the Beck--Fiala setting. We expect
that the lower bound construction in \cite{optimal-online}
can be adapted to the Beck--Fiala setting.}, it is known that there is an asymptotic gap between the discrepancy of worst--case inputs in the online versus offline problems \cite{optimal-online}.

A line of research spanning both the online and offline settings considers random and semi--random choices of the vectors $X_i$. Random discrepancy has garnered intense recent interest for a wealth of reasons, including interest in beyond--worst--case algorithmic aspects \cite{bansal-smooth1,bansal-smooth2}, connections with statistical physics \cite{APZ}, and perturbative analogues of long--standing conjectures in discrepancy theory \cite{bansal-meka}. Estimating the algorithmic and existential thresholds for discrepancy of random vectors, especially with regards to distinguishing the online and offline settings, remains an active and compelling area. We refer the reader to the ICM survey of Bansal \cite{bansal-survey} for a modern, comprehensive account of all the previously mentioned topics.

The present work investigates the online discrepancy of a sequence of uniformly random $d$--sparse binary vectors. The special case of $d = 2$, known as the ``online edge orientation'' or ``carpool'' problem, has a rich literature and much is known. We will be interested in the case of generic $d$, corresponding to the \textit{average--case online Beck--Fiala} problem. Letting $A \in \{0,1\}^{n \times T}$ be the matrix obtained by concatenating $T$ independent uniform $d$--sparse binary vectors, it is known in the \textit{offline} setting that $\disc(A) = \Theta(\sqrt{d})$ with high probability for any $d$, where the upper bound is achievable by an efficient algorithm \cite{bansal-meka}. 

We now present a simplified form of our main result, which shows a striking distinction between the online and offline guarantees. The full statements are briefly deferred. 
\begin{theorem}
    Let $T = n$, and let $X_1,\dots,X_T$ be independent and uniformly random $d$--sparse binary vectors. Then, for any $2 \le d \le n/2$, any online algorithm for assigning $\sigma \in \{-1,+1\}^T$ satisfies: w.h.p.,
    \[
        \big\|\sum_{i \in [T]} X_i\sigma_i\big\|_\infty \ge \max\Big(\frac{1}{8}\log\log n,\, \Omega\big(\sqrt{d}\big) \,\Big).
    \]
    On the other hand, for $d \le (\log\log n)^2 / \log\log\log n$, there is an explicit near-linear time online algorithm for assigning $\sigma$, satisfying: w.h.p.,
    \[
    \max_{t \in [T]}\,\big\|\sum_{i \in [t]} X_i\sigma_i\big\|_\infty \le 35\,\log\log n.
    \]
\end{theorem}

Some remarks are in order. 

\begin{itemize}
    \item Our result shows that for $d \le (\log\log n)^2/\log\log\log n$, the online discrepancy of the considered vectors is $\Theta(\log \log n)$. Surprisingly, the discrepancy is independent of the sparsity $d$ and thus the norm of the column vectors. In addition to the intrinsic interest of this threshold, our result also directly contrasts with the Beck--Fiala and Koml\'os conjectures, which respectively assert that sparsity and column norms are the driving mechanisms of offline discrepancy. 
    \item Previous constructions for lower bounds on online discrepancy consider the Koml\'os setting and heavily rely on block structure in addition to the usage of signed entries. In contrast, we consider ``mean--field'' (delocalized) matrix ensembles with binary, i.e., non--negative, entries. Our lower bound construction is conceptually different and uncovers new obstacles to
    achieving the offline discrepancy bound.
    \item Specifically for $d = 2$, corresponding to the average-case of the well-studied ``online edge orientation'' problem, a similar result was obtained in \cite{ajtai-carpool}.
    \item Our random construction establishes an asymptotic gap between the average online and offline Beck--Fiala settings.
    \item The given upper bound actually considers \textit{prefix} discrepancy, which is significantly stronger than bounding online discrepancy.
    \item The main result also extends to other time horizons, which we now expand on:
\end{itemize}

\begin{remark}[Extensions to other time horizons: upper bound]
    The upper bound extends by trivial modifications to any $T = \Theta(n)$, up to replacing the constant 35 with some constant depending on $T/n$. As our algorithm has guarantees on \textit{prefix} online discrepancy, inputs of size $T < n$ can be handled by padding. Inputs of size $T \ge n$ are dealt with by concatenating instances of length at most $n$. Thus, no generality will be lost when we later define our algorithm specifically for $T = n$.
\end{remark}

\begin{remark}[Extensions to other time horizons: lower bound]
    The lower bound of $\Omega(\sqrt{d})$ follows from the corresponding lower bound for the offline setting in the regime $T = \Theta(n)$, obtainable by the first moment method \cite{hypergraph}. The lower bound of $\log\log n$, which is the threshold focused on in this article, extends to arbitrary time horizons $T \ge n$. The technical statement is given in \cref{thm:lb}.

    It is also a somewhat surprising feature of our proof  that our construction for the lower bound accumulates large discrepancy very fast for early times $t \ll n$, but the growth in discrepancy decays double--exponentially in $t$. As a result, our proof of \cref{thm:lb} actually yields a lower bound of $c_\alpha \log\log n$, for $T \ge n/e^{\log(n)^\alpha}$ for any $\alpha < 1$, where $c_\alpha$ is some constant depending only on $\alpha$. 
\end{remark}

We now formally introduce some notation and then provide a technical restatement of our result. 

\subsection{Notation}
We begin by defining {\it online discrepancy}
in the setting of i.i.d arrivals.
Let $\mu$ be a probability measure in $\RR^n$, let
$T$ be an integer parameter, and let $(X_i)_{i \in [T]}$ be i.i.d draws from $\mu$. 

\begin{definition}[Online algorithm]
    An online algorithm is a random variable
$\sigma=(\sigma_1,\dots,\sigma_T)$ taking values in $\{-1,+1\}^T$ such that, 
for some auxiliary random variable $\xi$ independent from $(X_i)_{i \in [T]}$ (representing the auxiliary randomness of the algorithm), 
$\sigma_t$ is measurable with respect to the algebra generated by $(\xi;X_1,\dots,X_t)$.
for all $1\leq t\leq T$.
The quantity
\begin{equation}\label{disc not minimized}
\max\limits_{1\leq t\leq T}\Big\|
\sum_{i \in [t]} \sigma_i X_i
\Big\|_\infty
\end{equation}
is the {\it online discrepancy}
for the sequence $X_1,\dots,X_T$ under the algorithm encoded by $\sigma$. The mentioned filtration will be denoted by:
\begin{equation}
    \calF_t = \sigma(\xi;X_1,\dots,X_t)\,.
\end{equation}
\end{definition}

Next, we introduce the ``mean--field'' model of sparse random matrices that we henceforth consider: 

\begin{definition}[Matrix model]\label{def:matrix model}
    Let $1\leq d\leq n$ be integer parameters. Let $\mu_d$ denote the uniform measure on $d$--sparse binary column vectors of length $n$. Denote by $A\sim \calM_{n,T,d}$ an $n\times T$ random matrix with i.i.d columns drawn from $\mu_d$. 
\end{definition}

(All subsequent proofs will be written in terms of $A \sim \calM_{n,T,d}$ rather than $(X_i)_i$ for notational convenience). Finally, we introduce some notation for standard operations on prefixes. 

\begin{definition}[Prefix operations]
Let $x=(x_1,\dots,x_n)$ and $y=(y_1,\dots,y_n)$
be two $n$--dimensional vectors, and let $1\leq k\leq n$.
We write 
$$
\langle x,y\rangle_{[k]} := \sum_{j=1}^k x_j y_j, \quad \langle x,y\rangle_{[0]}:=0
$$
and
\[
    \supp_{[k]}(x) := \{i \in [k]\,:\,x_i \neq 0\}, \quad \supp_{[0]}(x) := \emptyset.
\]

\end{definition}

\subsection{Technical statement of results}

The first result deals with a lower bound on the online discrepancy. As previously remarked, the lower bound of $\Omega(\sqrt{d})$ for $T = \Theta(n)$ was previously known. Our focus is the following new threshold: 

\begin{theorem}[Lower bound]\label{thm:lb}
Let $n$ be a large integer, let $2 \le d \le n/2$, and let $T\geq n$.
Then for {\it any} online assignment rule $\sigma=(\sigma_1,\dots,\sigma_T)$
for the i.i.d arrivals $X_1,X_2,\dots,X_T$ distributed according to $\mu_d$,
the discrepancy \eqref{disc not minimized}
is at least $\frac{1}{2}c_{\text{\tiny\ref{thm:lb}}}\log\log n$ with probability $1-o(1)$, where $c_\lb = \frac{1}{4}$. 
\end{theorem}

As a counterpart of the result, we show that our lower bound
is optimal for the range of $d$ in which $\log\log n \ge \widetilde{\Omega}(\sqrt{d})$.

\begin{theorem}[Upper bound]\label{thm:ub}
Let $T = n$ and $n$ be a large integer.
Let $2\leq d\leq \frac{(\log \log n)^2}{\log\log\log n}$.
Assume that $X_1,\dots,X_T$ are i.i.d distributed according to $\mu_d$.
Then there is an online algorithm
producing signs $\sigma=(\sigma_1,\dots,\sigma_T)$,
such that with probability $1-o(1)$,
the prefix discrepancy \eqref{disc not minimized}
is bounded above by $C_{\text{\tiny\ref{thm:ub}}}\,\log\log n$,
where $C_{\text{\tiny\ref{thm:ub}}} = 35$.
\end{theorem}

The choice of $\sigma$ achieving the upper bound in Theorem~\ref{thm:ub}
is given below as Algorithm~\ref{alg}.
This online algorithm is implementable in time and storage complexity that is near--linear in the input size (i.e., in the parameter $nd$).

\begin{algorithm}
\caption{Online algorithm for discrepancy minimization of ultra--sparse random matrices, $T=n$}\label{alg}
\begin{algorithmic}
\State Set $T \gets n$ 
\State Set $C_\alg \gets 28$ \Comment{A sufficiently large constant} 
\State Sample $A \sim \mathcal{M}_{n,T,d}$
\State Sample $\widetilde \sigma \sim \mathrm{uniform}\{-1,+1\}^T$ \Comment{Random algorithm seed; will be used in analysis}

\vspace{.5em}

\State $\bfE_0 \gets \emptyset$ \Comment{List of rows with ``exceptional'' running products}

\For{$t \in [T]$}
    \State Reveal $\mathrm{col}_t(A)$
    \If{there is $i_*\in \supp(\mathrm{col}_t(A))\cap \bfE_{t-1}$
    such that
    \Statex \hspace{\algorithmicindent}$|\la \mathrm{row}_{i_*}(A),\, \sigma \ra_{[t-1]}|>
    |\la \mathrm{row}_{i}(A),\, \sigma \ra_{[t-1]}|$ for all $i\in
    \supp(\mathrm{col}_t(A))\setminus\{i_*\}$}
        \State $\sigma_t \gets -\mathrm{sign}\big(\la \row_{i_*}(A),\, \sigma\ra_{[t-1]}\big)$ \Comment{Note: $|\langle \mathrm{row}_{i_*}(A),\, \sigma \rangle_{[t]}| <|\la \mathrm{row}_{i_*}(A),\, \sigma \ra_{[t-1]}|$}
    \Else
        \State $\sigma_t \gets \widetilde \sigma_t$
    \EndIf

    \State $\bfE_{t} \gets \bfE_{t-1} \cup \{i\,:\, |\la \mathrm{row}_i(A),\, \sigma \ra_{[t]}| \ge C _\alg \log\log n \}$ \Comment{(Exceptional rows stay exceptional always)}
\EndFor

\vspace{.5em}
\State \textbf{Ensure:} $\max\limits_{t\leq T}\|\sum_{s\leq t}\sigma_s\col_s(A)\|_\infty \le (C_\alg+7) \log\log n$ \Comment{That is, $C_{\ub} := C_\alg + 7$.}
\State \textbf{Return:} $\sigma$
\end{algorithmic}
\end{algorithm} 

In words, our algorithm is: given the $t$'th column, check if there is a row supported in this column with both (1) exceptionally large prefix discrepancy, and (2) the \textit{strictly} largest current discrepancy among all rows supported in the column. If these conditions are met, we select $\sigma_t \in \{-1,+1\}$ to improve the balance of this row. Otherwise, we simply give up and assign $\sigma_t$ randomly.

\subsection{Previous work} The study of online discrepancy was initiated by Spencer \cite{spencer-online} in 1977. We refer the reader to the excellent ICM survey of Bansal \cite{bansal-survey} for a general account of the algorithmic developments in discrepancy theory, both offline and online. 

Regarding gaps between online and offline discrepancy, a work of Kulkarni, Reis, and Rothvoss \cite{optimal-online} established the dimension--dependent lower bound of $\sqrt{\log T}$ for online discrepancy in the Koml\'os setting\footnote{That is, for balancing $T$ vectors of unit Euclidean norm.} (in contrast with the conjectured dimension--free discrepancy of Koml\'os matrices in the offline setting). Their lower bound for online Koml\'os is provided by constructing a sparse block--diagonal matrix populated by coupled random signs. This work also develops an algorithm---which does not run in polynomial time---that achieves discrepancy matching the lower bound, establishing the (existential) asymptotic threshold for the worst--case online Koml\'os instances. The threshold for efficient algorithms remains open. 

Towards studying the \textit{average case} of online discrepancy minimization, Bansal et al \cite{bansal-online} consider the general setting of iid arrivals from arbitrary distributions, but sharp estimates remain open. Independent arrivals from dense, homogeneous distributions (such as uniform Rademacher vectors, in the average--case online Spencer setting), are quite well understood \cite{ALS2,bansal-meka,kim-roche}.

In the average case offline Koml\'os setting, one natural model is to take a matrix of iid random Gaussians.  This model was introduced in the statistical physics literature as a ``symmetric binary perceptron'' \cite{APZ} and has since been exactly solved  \cite{ALS1,dja-crit,PX,ss}. It turns out that the Gaussian model has the same behavior as matrices with iid Rademacher entries, and thus is far closer to modelling the average--case of Spencer's setting \cite{spencer1985six} rather than the Koml\'os setting.

More relevant to our setting, sparse random binary matrices have been used to model the average case offline Beck--Fiala setting. The Beck--Fiala conjecture asserts that any binary matrix with each column having support of at most $d$ has discrepancy $O(\sqrt{d})$. For random matrices with iid columns from $\mu_d$, Bansal and Meka \cite{bansal-meka} established the validity of the Beck--Fiala conjecture as well as provided a matching efficient algorithm for such matrices. Our result is in direct contrast: the Beck--Fiala conjecture fails for the {\it online} setting, even in the average case. There is also a significant body of work on the offline discrepancy of sparse binary matrices with divergingly wide aspect ratios. Interestingly, these matrices also have discrepancy that is independent of sparsity (see \cite{dja-jnw} and references therein).

A particular sub-setting of note within the online Beck-Fiala problem is the case of $d = 2$, known as the carpool or online edge orientation problem. For the average case of this problem, in which uniformly random $2$-sparse vectors must be signed online, a work of Ajtai et al. \cite{ajtai-carpool} established an analog of our main theorem for $d = 2$. For the worst case carpool problem, it also shown in \cite{ajtai-carpool} that the discrepancy is at least $\log(n)^{1/3}$, which falls short of the $\log(n)^{1/2}$ lower bound of \cite{optimal-online}. There is a substantial literature on various settings of the carpool problem; for modern developments, we refer to the discussions in \cite{expander-carpool,optimal-online}.

Finally, in the setting of offline worst--case Beck--Fiala and Komlos settings, an extremely recent landmark breakthrough of Bansal and Jiang \cite{bansal-jiang2,bansal-jiang1} established the Beck--Fiala conjecture under a sparsity lower bound: any matrix $M \in \{0,1\}^{n \times n}$ with $d$--sparse columns has discrepancy $O(\sqrt{d})$ as long as $d \ge (\log n)^2$. This, as well as a potential full resolution of the Beck--Fiala or Koml\'os conjecture, is not directly related to our results. Their upper bound is also implementable with an efficient algorithm.

\subsection{Discussion}
In order to concisely formulate some open problems, we introduce
the following notion.
\begin{definition}[Optimal expected online discrepancy]\label{def:online}
Let $\mu$ be a probability distribution in $\RR^n$,
and $T$ a time horizon.
Define the {\it optimal expected online discrepancy} as
$$
\online(\mu,T):=\inf\limits_{\sigma=(\sigma_1,\dots,\sigma_T)}\Exp\Big[\,\max\limits_{t\leq T}
\Big\|\sum_{s\in [t]} \sigma_s X_s\Big\|_\infty \Big],
$$
where $(X_1,\dots,X_T)\sim \mu^{\otimes T}$, and the infimum is taken over all online algorithms $\sigma=(\sigma_1,\dots,\sigma_T)$ (as in \cref{def:online}).
\end{definition}

From the main results of this note we see that, with $\mu_d$ as given in \cref{def:matrix model},
$$
\online(\mu_d,n)=\Theta(\log\log n),
$$
whenever $2\leq d\leq \frac{(\log\log n)^2}{\log\log\log n}$.
Our estimates leave open the problem of
identifying the optimal online discrepancy
in the regime of parameters $d=\Omega((\log \log n)^2)$.
We make the following conjecture:
\begin{conjecture}
For all $2\leq d\leq n/2$,
$$
\online(\mu_d,n)=\Theta(\max(\sqrt{d},\log\log n).
$$
\end{conjecture}
Note that the lower bound in the above conjecture is given 
in Theorem~\ref{thm:lb}, and the problem is to establish
a matching upper bound (which we believe should be attainable by an efficient algorithm). One can further ask what the optimal online discrepancy
is, viewed both as a function of $d$ and $T$, for arbitrary time horizon $T=\Omega(n)$. In particular, it is unclear what dependence to expect on $T$ (c.f. \cite{bansal-spencer}, in which Bansal and Spencer establish an online algorithm for random dense matrices of signs achieving $\sqrt{n}$ discrepancy for $T = e^{o(n)}$.)

\begin{question}
    What is the asymptotic dependence of $\online(\mu_d,T)$ on $T$?
\end{question}

\medskip

\noindent {\bf Acknowledgments.} K.T. 
is partially supported by NSF grant DMS 2331037. We thank V. Reis for bringing \cite{ajtai-carpool} to our attention.

\section{Lower bound}

This section is dedicated to the proof of \cref{thm:lb}. We begin by assuming $T = n$ and provide a simple reduction argument for arbitrary $T \ge n$ at the end of the section. Our starting goal is to establish that any online algorithm $\sigma$ satisfies:
\begin{equation*}
\mbox{for }T = n\;\mbox{and}\;2\leq d \le (\log\log n)^2,\;\;
\mbox{with high probability, }
    \|A\sigma\|_\infty \ge c_\lb \sqrt{\log\log n}.
\end{equation*}
For the sake of the eventual treatment of arbitrary $T \ge n$, we will actually prove the following strengthening:
\begin{equation}\label{eq:sparse-LB}
\mbox{for }T = n\;\mbox{and}\;2\leq d \le (\log\log n)^2,\;\;
\mbox{with high probability, }
    \big|\big\{i\,:\,\big|(A\sigma)_i\big| \ge c_\lb \sqrt{\log\log n}\big\}\big| > n^{1/4}.
\end{equation}

The main idea of the proof is to keep track
of pairs of rows with different magnitudes of partial products
and use the fact that, for certain realizations of a $t$--th column,
either choice of sign $\sigma_t$ leads to a widening
gap between the partial products.
As a toy example, suppose that after $t-1$ steps
of the algorithm we are given four rows $\row_{i_u}(A)$, $u=1,2,3,4$,
such that for some numbers $\ell\leq r$ we have
\begin{align*}
&\langle \row_{i_u}(A),\sigma\rangle_{[t-1]}=\ell,\quad u=1,2;\\
&\langle \row_{i_u}(A),\sigma\rangle_{[t-1]}=r,\quad u=3,4.
\end{align*}
Restricted to those four rows, the largest difference
between the magnitudes of partial products is $r-\ell$.
Next, assume that support of the $t$--th column of $A$ contains
$i_1$ and $i_3$ and does not contain $i_2$ and $i_4$.
Then, regardless of the choice of $\sigma_t$,
$$
\max\limits_{u=1,\dots,4}\langle \row_{i_u}(A),\sigma\rangle_{[t]}
-
\min\limits_{u=1,\dots,4}\langle \row_{i_u}(A),\sigma\rangle_{[t]}
=r-\ell+1,
$$
i.e the gap between the extreme values of the partial products grows.
For a rigorous implementation of the above idea and
towards establishing \eqref{eq:sparse-LB}, the key notion is the following: 
\begin{definition}
Let $(\ell,r,k,s)$ be integers satisfying $\ell\leq 0\leq r$,
$0\leq k\leq n$, and $s\geq 0$. Say that 
the pair $(A,\sigma)$ ``contains an $(\ell,r)$--spread
of size at least $s$ at time $k$'' if there exist collections of indices 
\[
    (i^{(\ell)}_j)_{j \in [s]} \quad \text{ and } \quad  (i^{(r)}_j)_{j \in [s]},
\]
such that: all $2s$ indices are distinct, and for every $1 \leq s' \le s$, it holds:
\[
    \langle \row_{i^{(\ell)}_{s'}}(A),\sigma\rangle_{[k]}=\ell; \quad \langle \row_{i^{(r)}_{s'}}(A),\sigma\rangle_{[k]}=r.
\]

\end{definition}

The main technical ingredient of the proof is the following lemma:
\begin{lemma}\label{lemma:lb spread}
Let $\Cspread := 3$. Fix any choice of online algorithm $\sigma$, let $1\leq q\leq \frac{\log\log n}{2\log \Cspread}$, and define
$$
k:=\sum_{u=1}^q \Big\lceil\frac{n}{(\log n)^{e^u}}\Big\rceil.
$$
Then with probability at least $1-\frac{q}{\log n}$ there are
$\ell\leq 0\leq r$ with $r-\ell\geq q$ such that 
$(A,\sigma)$ contains an $(\ell,r)$--spread
of size at least $\frac{n}{(\log n)^{\Cspread^q}}$ at time $k$.
\end{lemma}

We begin by showing that the above lemma implies \eqref{eq:sparse-LB}.

\begin{proof}[Proof of \eqref{eq:sparse-LB}] 
Define the following parameters: 
\begin{equation}\label{eq:LB params}
    q_0:=\Big\lfloor \frac{\log\log n}
{2\log \Cspread}\Big\rfloor; \quad k_0:=\sum_{u=1}^{q_0} \Big\lceil\frac{n}{(\log n)^{e^u}}\Big\rceil; \quad s_0 := \frac{n}{(\log n)^{\Cspread^{q_0}}}.
\end{equation}
Then Lemma~\ref{lemma:lb spread} implies that with probability at least $1-\frac{q_0}{\log n}$, the pair
$(A,\sigma)$ contains an $(\ell,r)$--spread
of size at least $s_0$ at time $k_0$, for some 
$\ell\leq 0\leq r$ with $r-\ell\geq q_0$. Define $\calE_{\text{spread}}$ as the $\calF_{k_0}$--measurable event that such a spread exists.  Define the $\calF_{k_0}$--measurable quantity
$$
I := \big\{i\in [n]\,:\, \big|\langle \row_{i}(A),
\sigma\rangle_{[k_0]}\big|\geq \frac{q_0}{2}\big\}.
$$
Note that conditioned on $\calE_{\text{spread}}$, it holds conditionally almost surely for $n$ sufficiently large, 
\[
    |I|\geq \frac{n}{(\log n)^{\Cspread^{q_0}}} \ge n\, e^{-O(\sqrt{\log n}\, \log\log n)} > \sqrt{n}.
\]
By independence of the columns of $A$, the $\calF_{k_0}$--measurability of $I$, and the construction of $(\sigma,\xi)$ as an ``online'' assignment, we then have
\begin{align*}
    \mathbb P\Big[\big\{\exists i\in I:\;a_{ij}=0\mbox{ for all }k_0<j\leq n\big\} \,\big|\,\calE_{\text{spread}},\,|I| = m\Big] &= 
    \mathbb P\Big[\bigcup_{i \in [m]}\{a_{ij}=0\mbox{ for all } j \in [n-k_0]\big\}\Big] \\
    &\le \mathbb P\Big[\bigcup_{i \in [m]}\{a_{ij}=0\mbox{ for all } j \in [n]\big\}\Big]\,.
\end{align*}
Towards bounding this last expression with a variance computation, let $Z_i$ be the indicator that the $i$'th row of $A$ is the zeroes vector. For any $m > \sqrt{n}$, recalling $d\le (\log\log n)^2 = o(\log n)$ by assumption, we compute:
\[
    \mathbb E\big[\sum_{i \in [m]} Z_i\big] = m\, \Big(1- \frac{d}{n}\Big)^{n} > \sqrt{n} \, e^{-2d} > 2\,n^{1/4}\,,
\]
and
\begin{align}
    \mathbb E\big[\big(\sum_{i \in [m]} Z_i\big)^2\big] &\le \mathbb E\big[\sum_{i \in [m]} Z_i\big] + m^2\, \Big(1- \frac{d}{n}\Big)^{n}\Big(1- \frac{d}{n-1}\Big)^{n} = \mathbb E\big[\sum_{i \in [m]} Z_i\big] + \mathbb E\big[\sum_{i \in [m]} Z_i\big]^2\Big(1+ O\Big(\frac{d}{n^2}\Big)\Big)^{n} \nonumber\\
    &\le \mathbb E\big[\sum_{i \in [m]} Z_i\big]^2 \big(1 + O\big(n^{-1/4}\big)\big)\,. \label{eq:second moment - empty rows}
\end{align}
In the final inequality, we have used that $(1 + O(d/n^2))^n = (1 + O(d/n)) = (1 + o(n^{-1/4}))$. 
Combining the above estimates and applying Chebyshev inequality yields:
\begin{align*}
    &\mathbb P\Big[\big|\big\{ i\in I:\;a_{ij}=0\mbox{ for all }k_0<j\leq n\big\}\big| > n^{1/4} \Big] \\
    &\quad\quad\ge \min_{m > \sqrt n} \, \Big(1 - \frac{q_0}{n}\Big)\,
    \mathbb P\Big[\big|\big\{ i\in I:\;a_{ij}=0\mbox{ for all }k_0<j\leq n\big\}\big| > n^{1/4} \,\big|\,\calE_{\text{spread}},\,|I| = m\Big] \\
    &\quad\quad\ge \Big(1 - \frac{q_0}{n}\Big)\Big(1 - O\big(n^{-1/4}\big)\Big) = 1 - o(1)\,.
\end{align*}
Thus, with probability $1-o(1)$,
there is are at least $n^{1/4}$ indices (in particular, at least one index) $i\leq n$ with
$|\langle \row_i(A),\sigma\rangle|\geq q_0/2$, whence we also have $\|A\sigma\|_\infty \ge q_0/2$. Taking $c_\lb := \big(8 \log \Cspread\big)^{-1} \ge 1/4 $ and recalling our choice of $q_0$ in \eqref{eq:LB params}, the lower bound \eqref{eq:sparse-LB} is established. 

\end{proof}

\begin{proof}[Proof of Lemma~\ref{lemma:lb spread}]
Recall $\Cspread = 3$. We proceed by induction on $q$.
Let $0\leq q\leq q_0$, where $q_0$ is given in \eqref{eq:LB params}, and consider the following inductive hypothesis, denoted by $\textbf{Hyp}_q$:

\begin{framed}
    $\textbf{Hyp}_q$: with probability at least $1-\frac{q}{\log n}$,
$(A,\sigma)$ contains an $(\ell,r)$--spread
of size at least $s_q$ at time $k_q$, where $\ell\leq 0\leq r$, $r-\ell\geq q$, and
\[
    s_q :=\left\lceil\frac{n}{(\log n)^{\Cspread^{q}}}\right\rceil; \quad k_q :=\sum_{u=1}^{q} \Big\lceil\frac{n}{(\log n)^{e^u}}\Big\rceil.
\]
(The $\calF_{k_q}$--measurable event that such a spread exists will be denoted by $\calE_q$.)
\end{framed}

Note that $\textbf{Hyp}_0$ is trivially satisfied since there are $n$ distinct rows with inner product $0$ at time $0$. In order to establish $\textbf{Hyp}_{q}$, it clearly suffices to show $\PP{\calE_{q}\,|\,\calE_{q-1}} > 1-1/\log n$. We turn towards this goal. For the remainder of the proof, assume that $\textbf{Hyp}_{q-1}$ has been established for some $1 \le q \le q_0$ and condition on any realization of $\calF_{k_{q-1}}$ within $\calE_{q-1}$. For brevity, we will also use the abbreviations
\[
    s:=s_{q-1}, \quad k:=k_{q-1}.
\]
(The parameters $s_q$ and $k_q$ will be denoted by explicit subscripts).  \\

For the remainder of this proof, we will frequently employ the following estimates on $s_q$ and $k_q - k$, which hold uniformly for $q \le q_0$ and all $n$ sufficiently large,
\begin{align}
    n^{.99} &\le n\, e^{-\sqrt{\log n}\, \log\log n}\le s_q \le \frac{n}{\log n} \label{eq:s estimate}\\
    n^{.99} &\le n\, e^{-(\log n)^{.4}\, \log\log n} \le k_q - k \label{eq:kq-k estimate}
\end{align}
Both estimates directly follow from our parameter choices.

By definition of an $(\ell,r)$--spread, $\textbf{Hyp}_{q-1}$ implies the existence of
two disjoint
$s$--subsets $I^{(\ell)}$ and $I^{(r)}$
of $[n]$ such that
    $$
    \langle \row_{i}(A),\sigma\rangle_{[k]}=\ell \quad \forall\,i\in I^{(\ell)}, \quad \text{and} \quad \langle \row_{i}(A),\sigma\rangle_{[k]}=r \quad \forall\, i\in I^{(r)}.
    $$
Let us further extract two technical estimates that enable the remainder of this inductive proof. 
\begin{proposition}\label{prop:lb col intersections}
The following two events each hold with $\calE_{q-1}$--conditional probability of at least $1-\frac{1}{\sqrt{n}}$:
\begin{align*}
\calE_a:=&\big\{
\mbox{For at least $s/2$ indices $i\in I^{(\ell)}$,
$a_{ij}=0$ for all $k<j\leq k_q$}
\big\}\;\cap\\
&\big\{
\mbox{For at least $s/2$ indices $i\in I^{(r)}$,
$a_{ij}=0$ for all $k<j\leq k_q$}
\big\},
\end{align*}
and 
\begin{align*}
\calE_b:=\Big\{
&\text{For at least $\frac{1}{100}\Big(\frac{s}{n}\Big)^2(k_q-k)$ indices $k<j\leq k_q$,}\\
&\mbox{$|\supp(\col_j(A))\cap I^{(\ell)}|=
|\supp(\col_j(A))\cap I^{(r)}|=1$, and}\\
&\mbox{$\big( \supp(\col_j(A)) \cap (I^{(\ell)} \cup I^{(r)})\big)\cap \bigcup\limits_{t\neq j,\,k<t\leq k_q}\supp(\col_t(A)) = \emptyset$}
\Big\}.
\end{align*}
\end{proposition}

\begin{proof}[Proof of Proposition \ref{prop:lb col intersections}]
The proof is based on a standard application of the second moment method.

    We begin by estimating the probability of $\calE_a$. Condition on any realization of the $\calF_{q-1}$--measurable random variable $I$ (where $I$ stands either for $I^{(\ell)}$ or $I^{(r)}$), within the event $\calE_{q-1}$. For each $i \in I$, let $Z_i$ denote the indicator of the event that $a_{ij} = 0$ for all $k<j\le k_q$. Note that $Z_i$ is conditionally independent of $\calF_{q-1}$. Thus, for any $q \ge 1$,
    \[
        \mu := \E{Z_i\,|\,\calF_{q-1},I,\calE_{q-1}} = (1 - d/n)^{k_q-k} \ge (1 - d/n)^{n/\log n} \ge 1 - o(1).
    \]
    We have used in the last inequality that $d = o(\log n)$ by assumption. Similarly, for any distinct  $i,i' \in I$,
    \[
        \E{Z_iZ_{i'}\,|\,\calF_{q-1},I,\calE_{q-1}} = \Big(1 - \frac{d}{n}\Big)^{k_q-k} \Big(1 - \frac{d}{n-1}\Big)^{k_q-k} = \mu^2 \Big(1- O\Big(\frac{d\,(k_q-k)}{n^2}\Big)\Big) = \mu^2(1 + o(1/n)).
    \]
    Combining these estimates, we obtain:
    \begin{align*}
        \mathrm{Var} \Big(\sum_{i \in I} Z_i - s\mu
        \,|\,\calF_{q-1},I,\calE_{q-1}\Big) = s\mu-s\mu^2+ O\Big(\frac{s^2}{n}\Big) = 
        o(s).
    \end{align*}
    Thus, recalling \eqref{eq:s estimate}, Chebyshev's inequality implies
    \begin{equation}\label{eq:Ea-estimate}
        \begin{split}
        \PP{\calE_a^c\,|\,\calF_{q-1},\calE_{q-1}} &\leq  2\,\PP{\sum_{i \in I} Z_i < s/2\,|\,\calF_{q-1},I,\calE_{q-1} }\\ &\le 2\,\PP{\big|\sum_{i \in I} Z_i - s\mu \big| > s/4
        \,|\,\calF_{q-1},I,\calE_{q-1}} 
        \le O\Big(\frac{1}{s}\Big)\,.
        \end{split}
    \end{equation}
    Turning now towards estimating the probability of $\calE_b$, for each $j \in (k,k_q]$, let $Y_j$ be the indicator of the event that $\supp(\col_j(A))$ has intersection of size exactly one with $I^{(\ell)}$ and $I^{(r)}$, denoted by $i^{(\ell)}_j\in I^{(\ell)}$
and $i^{(r)}_j\in I^{(r)}$ respectively, and furthermore that neither $i^{(\ell)}_j$ nor $i^{(r)}_j$ is in the support of the $t$'th column for any $t \in (k,k_q] \setminus \{j\}$. (In other words, $Y_j$ is the event that $j$ contributes to $\calE_b$). We claim:
    \begin{align*}
        \nu := \E{Y_j\,\big|\,\calF_k \cap \calE_{q-1}} &= \frac{\binom{n-2s}{d-2}\binom{s}{1}^2}{ \binom{n}{d}} \cdot \Bigg(\frac{\binom{n-2}{d}}{\binom{n}{d}}\Bigg)^{k_q-k - 1}=(1\pm o(1))\,d(d-1)\Big(\frac{s}{n}\Big)^2\,.
    \end{align*}
    Towards verifying the first identity, begin by recalling that $I^{(\ell)}$, $I^{(r)}$, and $\calE_{q-1}$ are $\calF_{k}$--measurable and $(a_{ij})_{i \in [n], j \in (k,k_q]}$ is independent of $\calF_k$. Next, by definition of $\calE_{q-1}$, we have that $I^{(\ell)}$ and $I^{(r)}$ are disjoint sets, both of size exactly $s$. Thus, the first ratio of binomial coefficients accounts for picking exactly one index from $I^{(\ell)}$ and $I^{(r)}$ each to be in the support of $\col_j(A)$, and the other $d-2$ entries in the column to be from $[n] \setminus (I^{(\ell)} \cap I^{(r)})$. The second term accounts for the conditional probability of 
    \[
        \supp(\,\col_t(A)\,) \cap \big\{i_j^{(\ell)},\, i_j^{(r)}\big\} = \emptyset, \quad \forall\,t \in (k,k_q]\setminus \{j\}.
    \]
    The second identity follows by expanding out the binomial coefficients and then applying \eqref{eq:kq-k estimate} as well as the assumption $d = o(\log n)$. The claim follows. Using \eqref{eq:kq-k estimate} again, for $n$ sufficiently large we also obtain:
    \[
        \E{\sum_{k < j \le k_q} Y_j\,\Big|\,\calF_k \cap \calE_{q-1}} = (k_q-k)\,\nu \ge n^{1-o(1)}\, n^{-.02} > n^{.9}\,.
    \]
    For any $j' \neq j$, similar counting 
    together with \eqref{eq:s estimate} and \eqref{eq:kq-k estimate}
    yield the second moment estimate:
    \begin{align*}
        \E{Y_jY_{j'}\,\big|\,\calF_k \cap \calE_{q-1}} &= \frac{\binom{n-2s}{d-2}\binom{s}{1}^2}{\binom{n}{d}} \cdot \frac{\binom{n-2s}{d-2}\binom{s-1}{1}^2}{\binom{n}{d}} \cdot \Big( \frac{\binom{n-4}{d}}{\binom{n}{d}} \Big)^{k_q-k-2}\\
        &= \nu^2\, \Big(1 + O\Big( \frac{d^2\cdot (k_q-q)}{n^2}+ \frac{1}{s}\Big)\Big) = \nu^2 \, \big(1 + o\big(n^{-.95}\big)\big). 
    \end{align*}
    From the first to the second line, we have used that:
    \[
        \frac{\binom{n-4}{d}}{\binom{n}{d}} = \Big(1 \pm O\Big(\frac{d^2}{n^2}\Big)\Big) \frac{\binom{n-2}{d}^2}{\binom{n}{d}^2}\,.
    \]
    In total, we have:
    \begin{align*}
        \mathrm{Var}
        \Big(\sum_{k < j \le k_q} Y_j\,\Big|\,\calF_k \cap \calE_{q-1}\Big)
        &=(k_q-k)(k_q-k-1)\nu^2 \, \big(1 + o\big(n^{-.95}\big)\big)+(k_q-k)\nu
        -(k_q-k)^2\nu^2\\
        &\le
        \mathbb E\Big[\sum_{k < j \le k_q} Y_j\Big]^2\,o(n^{-.9})\,.
    \end{align*}
    Applying Chebyshev's inequality yields the desired estimate:
    \begin{equation}\label{eq:Eb-estimate}
        \PP{\calE_b^c} \le n^{-.8}
    \end{equation}
    Together, \eqref{eq:Ea-estimate} and \eqref{eq:Eb-estimate} complete the proposition.
\end{proof}

With Proposition \ref{prop:lb col intersections} established, let us complete the proof of $\textbf{Hyp}_q$. 
Condition on any atom of $\calF_{q}$ such that
$\calE_{q-1} \cap \calE_a \cap \calE_b$ holds.
Let $J\subset(k,k_q]$ be the set of all column indices $j$ from the definition
of $\calE_b$, so that $|J|\geq \frac{1}{100}(s/n)^2(k_q-k)$. For every $j\in J$ denote by $i^{(\ell)}_j\in I^{(\ell)}$
and $i^{(r)}_j\in I^{(r)}$ the indices from the intersection
of $\supp\,\col_j(A)$ with $I^{(\ell)}$ and $I^{(r)}$,
respectively.
Crucially, by  definition of $\calE_b$,
the following holds conditionally almost surely within the event $\calE_{q-1} \cap \calE_{a} \cap \calE_b$.
For every $j\in J$, either:
\[
    \langle \row_{i^{(r)}_j}(A),\sigma\rangle_{[k_q]}=r+1, \quad \text{ or } \quad \langle \row_{i^{(\ell)}_j}(A),\sigma\rangle_{[k_q]}=\ell-1.
\]
Therefore, at least one of the following two sets must have size at least $|J|/2$:
\begin{enumerate}
    \item $\big\{ i \in [n]\,:\, \langle \row_{i}(A),\sigma\rangle_{[k_q]}=r+1\}$
    \item $\big\{ i \in [n]\,:\, \langle \row_{i}(A),\sigma\rangle_{[k_q]}=\ell - 1\}$.
\end{enumerate}
In the first case, using the definition of the event $\calE_a$,
we find at least $\min(|J|/2,s/2)$ disjoint pairs of indices
$(\tilde i^{(\ell)}_t,\tilde i^{(r)}_t)$, for $1\leq t \leq \lceil\min(|J|/2,s/2)\rceil$,
with 
\[
    \langle \row_{\tilde i^{(r)}_t}(A),\sigma\rangle_{[k_q]}=r+1,\quad
    \langle \row_{\tilde i^{(\ell)}_r}(A),\sigma\rangle_{[k]}=\ell,\quad
    1\leq r\leq \lceil\min(|J|/2,s/2)\rceil.
\]
Since we have chosen $\Cspread = 3 > e$, it holds for $n$ sufficiently large
\[
    \min\Big(\frac{|J|}{2}, \frac s2\Big)
    =\frac{1}{200}\Big(\frac{s}{n}\Big)^2(k_q-k)
    \geq \frac{1}{200}\cdot
    \frac{n}{(\log n)^{e^{q-1}}}\cdot
    \frac{1}{(\log n)^{2\cdot \Cspread^{q-1}}}
    \geq \left\lceil \frac{n}{(\log n)^{\Cspread^q}}\right\rceil = s_q.
\]
Thus, in the first case, $\calE_{q}$ holds. Similarly, in the second case, there are $2s_q$ distinct indices
$(\tilde i^{(\ell)}_t,\tilde i^{(r)}_t)_{t \in [s_q]}$ with
\[
    \langle \row_{\tilde i^{(r)}_t}(A),\sigma\rangle_{[k_q]}=r,\quad
    \langle \row_{\tilde i^{(\ell)}_r}(A),\sigma\rangle_{[k]}=\ell-1,\quad
    1\leq r\leq s_q,
\]
which also implies $\calE_q$. 
Thus, conditioned on $\calE_{q-1}\cap \calE_a\cap \calE_b$,
the pair $(A,\sigma)$ almost surely contains an $(\ell',r')$--spread
of size at least $s_q$ 
at time $k_q$, for some  $\ell' \leq 0\leq r'$
satisfying $r'-\ell'=r-\ell+1\geq q$. By the tail estimates in \cref{prop:lb col intersections}, we then have the desired inequality:
\[
    \PP{\calE_{q}\,|\,\calE_{q-1}} \ge 1 - \PP{\calE_a^c \cup \calE_b^c\,|\,\calE_{q-1}} \ge 1 - o\Big(\frac{1}{\log n}\Big)\,.
\]
This implies $\textbf{Hyp}_q$, completing the induction and thus the proof of \cref{lemma:lb spread}.
\end{proof}

\subsection{Arbitrary time horizon $T \ge n$}
It remains only to show the reduction of \cref{thm:lb} to \eqref{eq:sparse-LB}.

\begin{proof}[Proof of \cref{thm:lb}]
    We will assume that $T>n$.
    Fix an online algorithm $\sigma$. 
    Let $B$ be the $n\times T$ matrix with columns $X_1,\dots,X_T$ from the statement of the theorem.
    Consider the partition $B = [A_0; A]$,
    where $A_0$ is $n\times (T-n)$ and $A$ is $n\times n$.
    Note that $A_0 \sim \calM_{n,T-n,d}$, $A\sim \calM_{n,n,d}$,
    and $A_0$ and $A$ are independent. Correspondingly, write $\sigma = (\tau_0, \tau)$ where $\tau_0$ is the first $T-n$ coordinates of $\sigma$ and $\tau$ is the remaining final $n$ coordinates of $\sigma$. Our goal is to bound the online discrepancy of $B$ under $\sigma$. Condition on the algebra generated by $\xi$ and $A_0$, which we denote by $\calF_{A_0}$. Define:
    \[
        L_0:=\Big\{i \in [n]\,:\,|(A_0\tau_0)_i| > \frac{1}{2}c_{\lb}\log\log n \Big\}
    \]
    and 
    \[
        L := \big\{i \in [n]\,:\,|(A\tau)_i| > c_{\lb}\log\log n\big\}.
    \]
    As follows from \eqref{eq:sparse-LB}, the event $\{|L| > n^{1/4}\}$ holds with high probability. We henceforth restrict to this event and consider separately the cases that $|L_0| > n^{1/4}$ and $|L_0| \le n^{1/4}$. Within the event $\{|L_0| \le n^{1/4}\}\cap\{|L| > n^{1/4}\}$ it holds almost surely by the pigeonhole principle that $|(B\sigma)_i| > \frac{1}{2}c_{\lb}\log\log n$ for some $i$ (namely, any $i \in L \setminus L_0$). Then the desired lower bound holds conditionally almost surely. 
    Next, condition on the event $\{|L_0| > n^{1/4}\}\cap\{|L| > n^{1/4}\}$.
    In this case, the independence of $A$ from $\calF_{A_0}$ combined with a simple application of Chebyshev's inequality (essentially identical to the computation at the end of the proof of \cref{lemma:lb spread}) yields: with conditional probability tending to one,
    \[
        \big|\big\{i \in L_0\,:\, a_{ij} = 0, \quad  \forall j \in[n]\big\}\big| > \Omega\big(n^{1/4}\big)(1 - d/n)^n.
    \]  
    Note that $(1-d/n)^n \ge e^{-2d/n} \ge n^{o(1)}$, so that the above set has size at least $n^{1/4-o(1)}$. Further, $|(B\sigma)_i| > \frac{1}{2}c_\lb\log\log n$ for any $i$ in the above set. Thus, in total, it holds unconditionally with high probability that $\|(B\sigma)\|_\infty \ge \frac{1}{2}c_\lb\log\log n$. The proof is complete. 
\end{proof}

\section{Upper bound:
estimates on exceptional rows}
The goal of this section is to develop \textit{a priori} bounds controlling the execution of Algorithm \ref{alg}, which will be used in the next section to establish \cref{thm:ub}. The main result of this section is the following:
\begin{lemma}[Few exceptional rows are encountered by the algorithm]\label{lemma:exceptional row bound} Let $\widetilde \sigma$, $\sigma$, $A$, and $(\bfE_t)_{t\in [T]}$ be as defined in Algorithm \ref{alg}. For any fixed constant $C_\excep \ge 5$, there exists some $C_\alg$ sufficiently large (depending only on $C_\excep$) such that for $2\leq d \le (\log \log n)^2 (\log\log\log n)^{-1}$, it holds with respect to the joint randomness of $A$ and $\widetilde{\sigma}$, 
    \[
        \PP{|\bfE_T| \le n\, d^{-C_\excep}} \ge 1 - \log(n)^{-\omega(1)}\,.
    \]
\end{lemma}

Without any attempt to optimize, it suffices for this article to select the following parameter choices:
\[
    C_{\excep} = 5, \quad C_{\alg} = 28\,.
\]
Towards proving the lemma, we introduce the following ``categories'' of exceptional rows, which we claim will with high probability form a decomposition of the exceptional sets.

\begin{definition}[``Categories'' of exceptional rows]
    For each $t \in [T]$, define the sets $\bfC_{t,w} \subset [n]$ inductively via:
    \[
        \bfC_{t,0} = \Big\{i \in [n]\,:\, \big|\big\la \mathrm{row}_i(A),\, \widetilde \sigma \big\rangle_{[s]}\big| > \frac{1}{2}\cdot C_\alg \, \log \log n, ~ \text{ for some $s \le t$} \Big\},
    \]
    and, for each $w \in \mathbb N$,
    \begin{align*}
        \bfC_{t,w} = \Big\{i \in [n]\,:\, &\exists\, \text{ distinct } j_1,j_2 \in \mathrm{supp}(\mathrm{row}_i(A))\cap[t] \, \text{ s.t. for both $u \in \{1,2\}$, } \\
        &\exists \, i_u \in \big(\bfC_{j_u - 1,w-1} \cap \,\supp(\mathrm{col}_{j_u}(A))\big)\setminus\{i\}\Big\}.
    \end{align*}
\end{definition}

\begin{remark}\label{rem:nested}
By construction,
everywhere on the probability space,
$\bfC_{t,w}\subset \bfC_{t',w}$ for all $w\geq 0$
and $t\leq t'$.
\end{remark}

We begin by estimating the sizes of the categories, $\bfC_{t,w}$. 

\begin{proposition}\label{prop:category tails}
    We have: 
    \begin{equation}\label{eq:C0 estimate}
        \PP{|\bfC_{T,0}| > n\, d^{-2C_{\excep}}\,} \le n^{-.9}\,,
    \end{equation}
    and, letting $\tau := \log_{\frac{3}{2}}\,\log_{e} \,d^{2C_\excep}$,
    \begin{equation}\label{eq:Cw estimate}
        \PP{\forall w \ge 0\,:\,|\bfC_{T,w}| \le n\,  \exp\Big(-\Big(\frac{3}{2}\Big)^{w+\tau}\Big)} = 1- \log(n)^{-\omega(1)}\,.
    \end{equation}
\end{proposition}

Let us begin by observing that \cref{lemma:exceptional row bound} follows readily  upon assuming \cref{prop:category tails}.
\begin{proof}[Proof of \cref{lemma:exceptional row bound}]
    We have from \eqref{eq:Cw estimate}: with probability at least $1 - \log(n)^{-\omega(1)}$, 
    \begin{equation}\label{eq:C size estimates}
        \Big|\bigcup_{w \ge 0} \bfC_{T,w}\Big| \le \sum_{w \ge 0} n\,e^{-(\frac{3}{2})^{w+\tau}}
        \le O\Big(n\, e^{-(\frac{3}{2})^{\tau}}\Big) \le n\,d^{-C_\excep}\,.
    \end{equation}
    Moreover, as $6 > 1/\log(3/2)$,  \eqref{eq:Cw estimate} also yields
    \[
        \PP{|\bfC_{T,w}| = 0,\, \forall \, w \ge 6\log\log n}   
        \ge 1 - \log(n)^{-\omega(1)}.
    \]
    We claim that restricted to the event $\calE := \{|\bfC_{T,w}| = 0,\, \forall \, w \ge 6\log\log n\}$, almost surely $\bfE_t \subset \bigcup_{w} \bfC_{t,w}$ for all $t \in [T]$. If this claim is established, then \eqref{eq:C size estimates} implies the desired upper tail on $|\bfE_{T}|$, completing \cref{lemma:exceptional row bound}. Assume for induction that the claim holds for some $t-1 \ge 0$. (The base case of $t-1 = 0$ is trivial, since $E_{t-1} = \emptyset$). Restrict to $\calE$ and let $i \in \bfE_t$; our goal is to show $i \in \bfC_{t,w}$ for some $w$. By the induction hypothesis, it suffices to assume $i \in \bfE_t \setminus \bfE_{t-1}$. Further, we are done if $i \in \bfC_{t,0}$, so assume otherwise. Then
    $$
    \big|\big\la \mathrm{row}_i(A),\, \widetilde \sigma \big\rangle_{[s]}\big| \leq \frac{1}{2}\cdot C_\alg \, \log \log n, ~ \text{ for all $s \le t$}.
    $$
    Since $i\notin \bfE_{t-1}$, we have
    $$
    |\la \mathrm{row}_i(A),\, \sigma \ra_{[s]}| < C _\alg \log\log n
    \quad \mbox{for all }s\leq t-1.
    $$
    On the other hand, $i\in \bfE_{t}$ by assumption, whence
    $$
    |\la \mathrm{row}_i(A),\, \sigma \ra_{[t]}| \geq C_\alg \log\log n.
    $$
    Combining the above assertions, we obtain
    \[
        \big|\big\la \mathrm{row}_i(A),\, \widetilde \sigma-\sigma \big\rangle_{[t]}\big| \ge \big|\big\la \mathrm{row}_i(A),\, \sigma \big\rangle_{[t]}\big| - \big|\big\la \mathrm{row}_i(A),\, \widetilde \sigma \big\rangle_{[t]}\big| \ge \frac{1}{2}\,C_\alg \log\log n.
    \]
    In particular, assuming $C_\alg \geq 28$,
    \begin{equation}\label{eq:oajbfsaljhfbsljfk}
        |\{j \in \supp(\row_i(A)) \cap [t] \,:\, \sigma_j \neq \widetilde{\sigma}_j\}| \geq \frac{C_\alg}{4}\, \log\log n \geq 7 \log\log n.
    \end{equation}
    By construction of \cref{alg}, for any $j \in [T]$,
    a necessary condition for $\sigma_j \neq \widetilde{\sigma}_j$ to hold
    is that \[\bfE_{j-1} \cap \supp(\col_j(A))\neq\emptyset.\]
    Then, by \eqref{eq:oajbfsaljhfbsljfk} and the assumption $i\notin\bfE_{t-1}$,
    we have  
    \[
        \big|\Xi_t\big| \ge 7 \log\log n, \quad \text{ where } \quad \Xi_t := \big\{j \in \supp(\row_i(A)) \cap [t]\,:\, \bfE_{j-1} \cap \big(\supp(\col_j(A)) \setminus \{i\}\big) \neq \emptyset \big\}.
    \]
    For each $j \in \Xi_t$, let $i_j$ be an arbitrary element of $\bfE_{j-1} \cap \big(\supp(\col_j(A)) \setminus \{i\}\big)$. Since $i_j \in \bfE_{j-1}$ and $j-1 <  t$, our inductive hypothesis implies that $i_j$ belongs to $\bfC_{j-1,w_j}$ for some $w_j$. But, within the event $\calE$, almost surely there are at most $\lceil 6\log\log n\rceil$ non--empty categories $\bfC_{T,w}$. In particular, recalling Remark \ref{rem:nested}, 
    \[
        |\{w\,:\, \bfC_{s,w} \neq \emptyset, \text{ for some } s \in [T]\}| \le \lceil 6\log\log n\rceil .
    \]
    Thus, recalling \eqref{eq:oajbfsaljhfbsljfk}, the pigeonhole principle implies that there exists some $w$ and at least two distinct indices $j_1,j_2 \in \Xi_t$ with $w=w_{j_1} = w_{j_2}$, and
    $$
    i_{j_u}\in
    \bfE_{j_u-1} \cap \big(\supp(\col_{j_u}(A)) \setminus \{i\}\big)
    \cap \bfC_{j_u-1,w},\quad u=1,2.
    $$
    It follows by definition that $i \in \bfC_{t,w+1}$. The induction is established, finishing the proof.
\end{proof}
\begin{remark}
Note the choice of the index $i_*$
in \cref{alg} is not relevant for the above argument. We only needed the fact that $\sigma_j \neq \widetilde{\sigma}_j$
implies $\bfE_{j-1} \cap \supp(\col_j(A))\neq\emptyset$.
\end{remark}

The remainder of this section is devoted to the proof of \cref{prop:category tails}, which we now give. 

\begin{proof}[Proof of Proposition \ref{prop:category tails}]
    For the remainder of this proof, it will be convenient to use the shorthand:
    \[
        S_i := \mathrm{supp}(\mathrm{row}_i(A)).
    \]
    We begin by establishing the estimate on $|\bfC_{0,w}|$ provided by \eqref{eq:C0 estimate}. \\
    
    \textbf{Estimate for category zero}: using the maximal version of Bernstein's inequality \cite{maximal-bernstein} and the independence of the columns of $A$, it holds for any $i \in [n]$ and any $y > 0$:
    \[
        \PP{\max_{t \in [T]}|\la \mathrm{row}_i(A),\, \widetilde\sigma \ra_{[t]}| > y} \le 2\, \exp\Big( - \frac{y^2}{2\,d\,\big(1 - \frac{d}{n}\big) + 2y}\Big)\,.
    \]
    In particular, for any $i \in [n]$, taking $C_\alg \ge 4\,\sqrt{C_{\excep}}$ (satisfied by our choices $C_\excep = 5$, $C_\alg = 28$), we have
    \begin{equation}\label{eq:mu0}
        \mu_0 := \PP{i \in \bfC_{T,0}} \le \exp\Big(-\frac{1}{9}\,\,C_{\alg}^2\,\log \log \log  n\Big) < d^{-\frac{1}{5}\,C_\alg^2} < d^{-3\,C_\excep}\,.
    \end{equation}
    Next, consider the variance of $|\bfC_{T,0}|$. Letting $i,i' \in [n]$ be distinct, it is easily computed that $\E{|S_i\cap S_{i'}|} = \Theta(n(d/n)^2)$. Applying Markov's inequality then yields
    \begin{align*}
        \PP{\{i,i'\} \subset \bfC_{T,0}} &\le O\Big(\frac{d^2}{n}\Big) + \PP{\{i,i'\} \subset \bfC_{T,0}, ~S_i \cap S_{i'} = \emptyset }.
    \end{align*}
    Next, note that conditioned on the $S_i$, $S_{i'}$ and the event $\{S_i \cap S_{i'} = \emptyset\}$---all three of which are measurable with respect to the algebra generated by $A$---the variables $(A\widetilde\sigma)_i$ and $(A\widetilde\sigma)_{i'}$ are conditionally independent functions of disjoint subsets of the coordinates of $\widetilde{\sigma}$ (and $\widetilde\sigma$ is independent of $A$). Thus, 
    \begin{align*}
        \PP{\{i,i'\} \subset \bfC_{T,0}, ~S_i \cap S_{i'} = \emptyset } 
        &= \sum_{s,s'}\PP{\{i,i'\} \subset \bfC_{T,0}, \,S_i \cap S_{i'} = \emptyset, \,S_i = s,\, S_{i'} = s' } \\
        &\le \sum_{s,s'}\PP{i \in \bfC_{T,0}, \,S_{i} = s}\PP{i' \in \bfC_{T,0}, \, S_{i'} = s'} \\
        &= \PP{i \in \bfC_{T,0}}^2\,.
    \end{align*}
    In summary, we have $\E{|\bfC_{T,0}|^2} \le \E{|\bfC_{T,0}|}^2 + \E{|\bfC_{T,0}|} + O(n\,d^2)$. Chebychev's inequality and \eqref{eq:mu0} then yield the desired tail \eqref{eq:C0 estimate} on the number of ``category zero'' exceptional rows:
    \[
        \PP{|\bfC_{T,0}| > n \,d^{-2\,C_{\excep}}\,} \le \frac{\mathrm{Var}(|\bfC_{T,0}|)}{n^2 d^{-4\,C_{\excep}} } \le n^{-.9}\,.
    \]

    \vspace{1em}
    
    \textbf{Estimate for remaining categories}: next, we turn towards developing \eqref{eq:Cw estimate}. Let $m$ be some integer parameter chosen later (corresponding to our estimate on $|\bfC_{T,w-1}|$), assumed to satisfy $m \le 2nd^{-2C_\excep}$. Note that for any $j \in [T]$, the event $\{|\bfC_{T,w-1}| \le m\}$ is contained within the event $\{|\bfC_{j,w-1}| \le m\}$. Let us decompose the event $\{i \in |\bfC_{T,w}|\} \cap \{|\bfC_{T,w-1}| \le m\}$ so as to exploit the independence of columns in $A$. Define the events:
    \begin{align*}
        \widetilde{\calE}_m &:= \{|\bfC_{T,w-1}| \le m\} \\
        \mathcal{E}_{j}(i) &:= \{i \in \supp(\mathrm{col}_{j}(A)),\, \bfC_{j - 1,w-1} \cap \,\supp(\mathrm{col}_{j}(A)) \setminus \{i\} \neq \emptyset,~|\bfC_{j-1,w-1}| \le m\}
    \end{align*}
    In this notation, we have:
    \begin{equation}\label{eq:Cw event containment}
        \{i \in \bfC_{T,w}\} \cap \widetilde{\calE}_m \subset \bigcup_{1 \le j < j' \le T} \mathcal{E}_{j}(i) \cap \mathcal{E}_{j'}(i)\,.
    \end{equation}
    Recalling the filtration $\calF_j$ from Definition \ref{def:online}, note that $\calE_{j}(i)$ is $\calF_j$ measurable. Furthermore, $|\bfC_{j-1,w-1}|$ is $\calF_{j-1}$--measurable and $\mathrm{supp}(\col_{j}(A))$ is independent of $\calF_{j-1}$. Using this independence as well as the assumed upper bound on $m$ (which, in particular, implies $md/n = o(1)$), it holds for $n$ sufficiently large
    \begin{align*}
        \PP{\widetilde{\calE}_m \cap \calE_{j}(i)\,\big|\,\calF_{j-1}} &\le \PP{|\bfC_{j-1,w-1}| \le m,\, i \in \supp(\mathrm{col}_{j}(A)),\, \bfC_{j - 1,w-1} \cap \,\supp(\mathrm{col}_{j}(A)) \setminus \{i\} \neq \emptyset \,\big|\,\calF_{j-1} } \nonumber\\
        &\le \frac{d}{n}\, \Big(1 - \Big(1 - \frac{m}{n}\Big)^{d-1}\Big) \nonumber\\
        &\le C\,m\,\Big(\frac{d}{n}\Big)^2\,,
    \end{align*}
    where $C > 1$ is a universal constant. 
    Thus,
    \begin{align}
        \PP{\widetilde{\calE}_m \cap \calE_{j}(i) \cap \calE_{j'}(i)} &\le  C\, m\,\Big(\frac{d}{n}\Big)^2 \cdot \PP{\widetilde{\calE}_m \cap \calE_{j'}(i)\,\big|\, \calE_{j}(i) \cap \calF_{j'-1}} \le \Big(C\, m\,\Big(\frac{d}{n}\Big)^2\Big)^2  \,. \label{eq:Ej estimate}
    \end{align}
    Relation \eqref{eq:Ej estimate}, combined with a union bound over $j$ and $j'$, yields for every $i$:
    $$
    \PP{\{i \in \bfC_{T,w}\} \cap \widetilde{\calE}_m}
    \leq |T|^2\,\Big(C\, m\,\Big(\frac{d}{n}\Big)^2\Big)^2  \,.
    $$
    Using Markov's inequality provides the desired tail on $\bfC_{T,w}$ in the case that $m$ is small. 
    The last relation
    yields: 
    \begin{align}
        \forall\, m \le n^{1/2}, \quad  \PP{\widetilde{\calE}_m, \, |\bfC_{T,w}|>n\,\Big(\frac{m}{n}\Big)^{3/2}} &\le 
        |T|^2\,n\,\Big(C\, m\,\Big(\frac{d}{n}\Big)^2\Big)^2
        \bigg(n\,\Big(\frac{m}{n}\Big)^{3/2}\bigg)^{-1}
        = \log(n)^{-\omega(1)}\,.\label{eq:small m - Cw tail}
    \end{align}

    Towards treating the case of large values of $m\ge n^{1/2}$, we use Markov's inequality applied to some large power of $|\bfC_{T,w}|$
    (i.e apply high moment methods as opposed to the first moment for small $m$). Fix distinct indices $i_1,\dots,i_r$ from $[n]$ (with $r$ chosen later).
    We are interested in estimating the probability 
    of the event
    $$
    \{i_1,i_2,\dots,i_r \in \bfC_{T,w}\} \cap \widetilde{\calE}_m.
    $$
    Similarly to \eqref{eq:Cw event containment}, we have
    \begin{equation}\label{eq:highmommentCTw}
    \PP{\{i_1,i_2,\dots,i_r \in \bfC_{T,w}\} \cap \widetilde{\calE}_m}
    \leq \sum\limits_{j_\ell < j_\ell',\,\ell \in [r]}
    \PP{\widetilde{\calE}_m \cap \bigcap_{\ell \in [r]} \calE_{j_\ell}(i_\ell)\cap \calE_{j_\ell'}(i_\ell) }.
    \end{equation}
    First, consider distinct indices $j_1,j_1',\dots,j_r,j_r'$ from $[T]$ with $j_1 < j_1' < j_2 < j_2' < \dots < j_r < j_r'$.
    Computing similarly to \eqref{eq:Ej estimate}, by sequentially conditioning on the the filtration $\calF_{j}$ at the increasing sequence of times $j$ given by $j_1,j_1',j_2,j_2',\dots$, we obtain:
    \begin{equation}\label{eq:Cw mean}
        \PP{\widetilde{\calE}_m \cap \bigcap_{\ell \in [r]} \calE_{j_\ell}(i_\ell)\cap \calE_{j_\ell'}(i_\ell) } \le \Big(C\,m\,\Big(\frac{d}{n}\Big)^2\Big)^{2r}\,.
    \end{equation}
    Observe that the above estimate is actually true for all distinct $j_1,j_1',\dots,j_r,j_r'$ satisfying $j_\ell < j_\ell'$ for $\ell \in [r]$.

    Next, we extend \eqref{eq:Cw mean} to the case where the indices $j_1,j_1',\dots,j_r,j_r'$ are not necessarily distinct.
    Assume that $j_\ell < j_\ell'$ for $\ell \in [r]$, 
    and let $D := \bigcup\limits_{\ell\in[r]}\{j_\ell,j_\ell'\}$.
    Then, a simple modification of the above yields
    \begin{equation}\label{eq:Cw moments}
        \PP{\widetilde{\calE}_m \cap \bigcap_{\ell \in [r]} \calE_{j_\ell}(i_\ell)\cap \calE_{j_\ell'}(i_\ell) } \le \Big(C\,\frac{d}{n}\Big)^{2r}\, \Big(m\,\frac{d}{n}\Big)^{|D|} \,,
    \end{equation}
    where the factor of $(d/n)^{2r}$ accounts for each row $i_1,\dots,i_r$ having support in two prescribed columns (distinct for each row, since $j_\ell < j_\ell'$ for all $\ell$), and the factor of $(md/n)^{|D|}$ arises in estimating the probability that some row in $\bfC_{j-1,w-1}$ has support in each distinct column $j \in D$.
    In view of \eqref{eq:highmommentCTw} and \eqref{eq:Cw moments},
    \begin{align}\label{eq:CTw indicesi1ir}
    \PP{\{i_1,i_2,\dots,i_r \in \bfC_{T,w}\} \cap \widetilde{\calE}_m}
    \leq\sum_{|D| \le 2r}n^{|D|}\,|D|^{2r-|D|}\Big(C\,\frac{d}{n}\Big)^{2r}\, \Big(m\,\frac{d}{n}\Big)^{|D|},
    \end{align}
    where the factor of $n^D|D|^{2r-|D|}$ crudely bounds the number of ways to select the $j$'s, such that there are $|D|$ total distinct column indices.

    Returning to estimating the size of $\bfC_{T,w}$,
    letting $r_0 = (\log\log n)^2$ , we have:
    \begin{align*}
        \PP{|\bfC_{T,w}| >  2\,C^2\,d^4\,\Big(\frac{m}{n}\Big)^2\,n} &\le \PP{\widetilde{\calE}_m^c} + \PP{|\bfC_{T,w}| >  2\,C^2\,d^4\,\Big(\frac{m}{n}\Big)^2\,n,\, \widetilde{\calE}_m} \\
        &\le \PP{\widetilde{\calE}_m^c} + \PP{|\bfC_{T,w}|^{r_0} >  2^{r_0} \Big(\,C^2\,d^4\,\Big(\frac{m}{n}\Big)^2\,n\Big)^{r_0},\, \widetilde{\calE}_m}.
    \end{align*}
    This last probability term will be bounded using Markov's inequality. The expectation of $|\bfC_{T,w}|^{r_0}\ind(\widetilde{\calE}_m)$ can be evaluated by applying \eqref{eq:CTw indicesi1ir}, then taking a union bound over all choices of distinct indices $\{i_1,\dots,i_r\} \subset [n]$ for each $r \le r_0$.  
    In total:
    \begin{align*}
        \forall m \ge n^{1/2}, \quad &\PP{|\bfC_{T,w}|^{r_0} >  2^{r_0} \Big(\,C^2\,d^4\,\Big(\frac{m}{n}\Big)^2\,n\Big)^{r_0},\, \widetilde{\calE}_m} \\
        &\le\frac{1}{2^{r_0} \Big(C^2\,d^4\,\Big(\frac{m}{n}\Big)^2\,n\Big)^{r_0}} \sum_{r \le r_0}r_0^{r_0-r}n^r\sum_{|D| \le 2r}n^{|D|}\,|D|^{2r-|D|}\,  \Big(C\frac{d}{n}\Big)^{2r}\, \Big(m\,\frac{d}{n}\Big)^{|D|}  \\
        &\le\frac{1}{2^{r_0} \Big(C^2\,d^4\,\Big(\frac{m}{n}\Big)^2\,n\Big)^{r_0}} \sum_{r \le r_0}O\bigg(r_0^{r_0-r}\,n^{3r}\,  \Big(C\frac{d}{n}\Big)^{2r}\, \Big(m\,\frac{d}{n}\Big)^{2r}\bigg)  \\
        &\le\frac{1}{2^{r_0} \Big(C^2\,d^4\,\Big(\frac{m}{n}\Big)^2\,n\Big)^{r_0}}\,\, O\Big(n^{3r_0} \Big(C\frac{d}{n}\Big)^{2r_0}\, \Big(m\,\frac{d}{n}\Big)^{2r_0} \Big)  \\
        &\le O\big(2^{-r_0}\big) .
    \end{align*}
    In the second line, the factor of $r_0^{r_0-r}n^{r}$ is a crude upper bound on the number of ways to select $i_1,\dots,i_{r_0}$
    in such a way that there are $r$ distinct indices. 
    From the second to the third line, we have used that $|D| \le 2(\log\log n)^2$ and $m \ge n^{1/2}$. 
    
    Recall that we have assumed $m$ satisfies $m/n \le d^{-2C_\excep}$, where $C_\excep$ is a sufficiently large constant. In particular, we can and will assume $C_\excep \geq 5$. Then,
    \begin{equation}\label{eq:Cw final inductive tail}
    \begin{split}
    \forall\,n^{1/2}\leq m\leq nd^{-2C_\excep},\quad
    \PP{\frac{|\bfC_{T,w}|}{n} >  \Big(\frac{m}{n}\Big)^{3/2},\, \widetilde{\calE}_m} 
    &\le \PP{\frac{|\bfC_{T,w}|}{n} >  2\,C^2\,d^4\,\Big(\frac{m}{n}\Big)^2,\, \widetilde{\calE}_m}\\
    &\le (\log n)^{-\omega(1)}.
    \end{split}
    \end{equation}
    Between \eqref{eq:Cw final inductive tail} and \eqref{eq:small m - Cw tail}, this estimate now holds for the full range of $m \le nd^{-2C_\excep}$. We are ready to select parameters and conclude. Define the sequence $(m_w)_{w \ge 0}$ by:
    \[
        m_{w} = \exp\Big(-\Big(\frac{3}{2}\Big)^{w+\tau}\Big)\,, \quad \text{ where } \quad 
        \tau = \log_{3/2}\Big( \log_{e} \Big(d^{2C_\excep}\Big)\Big)\,,
    \]
    so that
    \[
        m_0 = \frac{1}{d^{2C_\excep}}\quad \text{ and } \quad m_{w+1}= (m_{w})^{\frac{3}{2}}.
    \]
    Note that $m_w < 1/n$ if $w > 5.7\,\log\log n$. On the event $\{\bfC_{T,w_0} = \emptyset\}$, conditionally almost surely $\bfC_{T,w} = \emptyset$ for all $w > w_0$ by construction. Thus, for any $w \ge 0$, it follows by \eqref{eq:small m - Cw tail}, \eqref{eq:Cw final inductive tail} (applied with $m:=m_{w-1}\,n$) and \eqref{eq:C0 estimate},
    \begin{align*}
        \PP{\forall w \ge 0\,:\,\frac{|\bfC_{T,w}|}{n} \le m_w} &= \PP{\forall\, w \in [6\log\log n]\,:\,\frac{|\bfC_{T,w}|}{n} \le m_w} \\
        &\ge 1 - \PP{|\bfC_{T,0}| > m_0} - \sum_{1 \le w \le 6 \log \log n}   \PP{\frac{|\bfC_{T,w}|}{n} > m_w,\, \frac{|\bfC_{T,w-1}|}{n} \le m_{w-1} } \\ 
        &\ge 1 - n^{-.9} - \sum_{1 \le w \le 6 \log \log n} \PP{\frac{|\bfC_{T,w}|}{n} > (m_{w-1})^{\frac{3}{2}},\, \frac{|\bfC_{T,w-1}|}{n} \le m_{w-1} } \\ 
        &\ge 1 - n^{-.9} - 6 \log \log n \cdot \log(n)^{-\omega(1)} \\
        &\ge 1 - \log(n)^{-\omega(1)}\,.
    \end{align*}
    This completes the proof of \cref{prop:category tails}. As already proven, \cref{lemma:exceptional row bound} immediately follows as well. 
\end{proof}

\section{Upper bound: analysis of the algorithm}

With \cref{lemma:exceptional row bound} established, this section is dedicated to completing \cref{thm:ub}. For every $t\in[T]$, and every integer $k\geq 0$
define the $\calF_t$--measurable random set
$$
\calM(t,k):=
\big\{
i\leq n:\;|\langle \row_i(A),\sigma\rangle_{[s]}|
\geq (C_\alg+1)\log\log n
+3k
\mbox{ for some }s\in[t]\big\}.
$$
Observe that for every $k$,
the sequence $(\calM(t,k))_{t \in [T]}$ is nested,
i.e
$$\calM(1,k)\subset \calM(2,k)\subset\dots\subset \calM(T,k).$$
Similarly, for every $t\in[T]$,
$$\calM(t,0)\supset \calM(t,1)\supset \calM(t,2)\supset\dots$$
Our goal is to show that $\calM(T,k)$ is empty with high probability when $k$ is chosen as a large multiple of $\log\log n$. Towards inductively estimating the decay of $|\calM(T,k)|$ as $k$ increases, we begin with the base case. 

\begin{lemma}\label{lem:stepzero}
The event
$$
\Event_{\text{\tiny\ref{lem:discub}}}(0):=
\Big\{
|\calM(T,0)|\leq \frac{n}{\log^2 n}
\Big\}
$$
has probability $1-O((\log n)^{-1})$.
\end{lemma}
\begin{proof}
Fix $i \in [n]$ and define the parameter 
\[  
    m = \lfloor \log\log n\rfloor.
\]
Within the event $\{i\in \calM(T,0)\}$, almost surely
there are at least $m$ indices
$1\leq j_1<\dots<j_m\leq T$
such that
\begin{equation}\label{eq:rowsumgrowthbase}
|\langle \row_i(A),\sigma\rangle_{[j_\ell]}|
>|\langle \row_i(A),\sigma\rangle_{[j_\ell-1]}|
\geq C_\alg\log\log n,\quad \forall \,\ell\in[m].
\end{equation}
For every $\ell\in[m]$, condition \eqref{eq:rowsumgrowthbase} implies
that, first, $i\in \supp(\col_{j_{\ell}}(A))\cap \bfE_{j_\ell-1}$,
and, second, 
$\sigma_{j_\ell} \neq -\mathrm{sign}\big(\la \row_{i}(A),\, \sigma\ra_{[j_\ell-1]}\big)$.
In view of the construction rules for $\sigma$ in \cref{alg}---specifically, the cases considered in the ``if/else'' clauses---there is some $i_\ell\in \supp(\col_{j_{\ell}}(A))\setminus\{i\}$
with
$$
|\langle \row_{i_\ell}(A),\sigma\rangle_{[j_\ell-1]}|
\geq 
|\langle \row_i(A),\sigma\rangle_{[j_\ell-1]}|\geq
C_\alg\log\log n.
$$
In particular, $i_\ell\in\bfE_{j_\ell-1}\setminus\{i\}$ and thus $\bfE_{j_\ell-1}\setminus\{i\}\neq
\emptyset$. 

Next, define the $\calF_{j_m}$--measurable event
\[
    \Event(j_1,\dots,j_m):=
\bigcap_{\ell\in[m]}\{i\in\supp\,\col_{j_\ell}(A)\}\cap \{
\supp\,\col_{j_\ell}(A)\cap (\bfE_{j_\ell-1}
\setminus\{i\})\neq\emptyset\}\cap \{
|\bfE_{j_\ell-1}| \le n\, d^{-C_\excep}\}.
\]
In this notation, a union bound yields
\begin{equation}\label{eq:union bound MT0}
    \Prob\big(\big\{i\in \calM(T,0)\big\}\cap 
\big\{|\bfE_T| \le n\, d^{-C_\excep}\big\}\big)
\leq
\sum\limits_{j_1<\dots<j_m}
\Prob\big(\Event(j_1,\dots,j_m)\big).
\end{equation}
Towards estimating the probability term on the right, we decompose $\Event(j_1,\dots,j_m)$ into a nested sequence of events,
$$
\Event_{j_1}\supset \Event_{j_1}'\supset
\Event_{j_2}\supset \Event_{j_2}'\supset\dots
\Event_{j_m}\supset \Event_{j_m}'=\Event(j_1,\dots,j_m),
$$
defined inductively by: 
\begin{align*}
\Event_{j_1} &:=\{|\bfE_{j_1-1}| \le n\, d^{-C_\excep}\} \\
\Event_{j_q}'&:=\Event_{j_q}\cap
\big\{i\in\supp\,\col_{j_q}(A)\} \cap \{
\supp\,\col_{j_q}(A)\cap (\bfE_{j_q-1}
\setminus\{i\})\neq\emptyset\big\},\;\;1\leq q\leq m;\\
\Event_{j_{q+1}}&:=\Event_{j_q}'\cap
\big\{|\bfE_{j_{q+1}-1}| \le n\, d^{-C_\excep}\big\},\;\;
1\leq q\leq m-1.
\end{align*}
Note that $\Event_{j_q}$ is
$\calF_{j_q-1}$--measurable,
and $\Event_{j_q}'$ is $\calF_{j_q}$--measurable, for all $1\leq q\leq m$.
Observe further that by independence of the columns of $A$, almost surely
\[
    \Prob\big(\calE_{j_{q}}'\,\big|\,\calF_{j_{q}-1} \cap \calE_{j_q}\big) = O\Big(\frac{d}{n}\cdot d\cdot d^{-C_\excep}\Big).
\]
Repeatedly utilizing the above inequality yields
\[
    \Prob\big(\Event(j_1,\dots,j_m)\big)
\leq \bigg(O\Big(\frac{d}{n}\cdot d\cdot d^{-C_\excep}\Big)
\bigg)^m.
\]
Combining this estimate with \eqref{eq:union bound MT0},
$$
\Prob\big(\big\{i\in \calM(T,0)\big\}\cap 
\big\{|\bfE_T| \le n\, d^{-C_\excep}\big\}\big)
=
\bigg(O\Big(\frac{d^2}{m}\cdot d^{-C_\excep}\Big)
\bigg)^m.
$$
Applying Markov's inequality,
$$
\Prob\Big(\Big\{
|\calM(T,0)|\geq \frac{n}{\log^2 n}
\Big\}\cap 
\big\{|\bfE_T| \le n\, d^{-C_\excep}\big\}\Big)
\leq \log^2 n\,\bigg(O\Big(\frac{d^2}{m}\cdot d^{-C_\excep}\Big)
\bigg)^m.
$$
Since $C_\excep\geq 2$, our definition of $m$
and Lemma~\ref{lemma:exceptional row bound}
yield the result.
\end{proof}

The next lemma provides the inductive step for bounding $|\calM(T,k)|$. The proof is essentially identical to the base case. 

\begin{lemma}\label{lem:discub}
For every $k\geq 1$,
the event
$$
\Event_{\text{\tiny\ref{lem:discub}}}(k):=
\big\{
|\calM(T,k)|
\le n\, (\log n)^{-2\cdot 2^k}
\big\}
$$
has probability at least $\Prob(\Event_{\text{\tiny\ref{lem:discub}}}(0))
-k\log^{-1}n$.
\end{lemma}
\begin{proof}
The proof closely mirrors that of Lemma~\ref{lem:stepzero}, so some details are omitted.
Assume by way of induction that the event $\Event_{\text{\tiny\ref{lem:discub}}}(k-1)$
has probability at least
$\Prob(\Event_{\text{\tiny\ref{lem:discub}}}(0))
-(k-1)(\log n)^{-1}$
(the base of the induction is trivial).
Fix $i\in [n]$. Within the event $\{i\in \calM(T,k)\}$, almost surely there are at least three indices $1\leq j_1<j_2<j_3\leq T$
such that
\begin{equation}\label{eq:rowsumgrowth}
|\langle \row_i(A),\sigma\rangle_{[j_\ell]}|
>|\langle \row_i(A),\sigma\rangle_{[j_\ell-1]}|
\geq (C_\alg+1)\log\log n+3(k-1),\quad \ell=1,2,3.
\end{equation}
For every $\ell\in[3]$, condition \eqref{eq:rowsumgrowth} implies
that, first, $i\in \supp(\col_{j_{\ell}}(A))\cap \bfE_{j_\ell-1}$,
and, second, 
$\sigma_{j_\ell} \neq -\mathrm{sign}\big(\la \row_{i}(A),\, \sigma\ra_{[j_\ell-1]}\big)$.
In view of the construction rules for $\sigma$ in \cref{alg},
there exists some $i_\ell\in \supp(\col_{j_{\ell}}(A))\setminus\{i\}$
with
$$
|\langle \row_{i_\ell}(A),\sigma\rangle_{[j_\ell-1]}|
\geq 
|\langle \row_i(A),\sigma\rangle_{[j_\ell-1]}|\geq
(C_\alg+1)\log\log n+3(k-1).
$$
In particular, $i_\ell\in\calM(j_\ell-1,k-1)\setminus\{i\}$ and $\calM(j_\ell-1,k-1)\setminus\{i\}\neq
\emptyset$. 

Define the $\calF_{j_3}$--measurable event
\[
\Event(j_1,j_2,j_3):= \bigcap_{\ell \in [3]}
\{i\in\supp\,\col_{j_\ell}(A)\}\cap \{
\supp\,\col_{j_\ell}(A)\cap (\calM(j_\ell-1,k-1)
\setminus\{i\})\neq\emptyset\}.
\]
In this notation, a union bound over the column indices $j_1,j_2,j_3$ yields
\[
\Prob\big(\big\{i\in \calM(T,k)\big\}\cap 
\Event_{\text{\tiny\ref{lem:discub}}}(k-1)\big)
\leq
\sum\limits_{j_1<j_2<j_3}
\Prob(\Event_{\text{\tiny\ref{lem:discub}}}(k-1)\,
\cap\,\Event(j_1,j_2,j_3)).
\]
Recalling the induction hypothesis and conditioning sequentially on $\calF_{j_1-1}$, $\calF_{j_2-1}$, and $\calF_{j_3-1}$, a computation similar to that in the proof of \cref{lem:stepzero} yields the estimate: 
\[
    \Prob\Big(\Event_{\text{\tiny\ref{lem:discub}}}(k-1)\, \cap\,\Event(j_1,j_2,j_3)\Big) = \bigg(O\Big(\frac{d}{n}\cdot d\cdot (\log n)^{-2\cdot 2^{k-1}}\Big)\bigg)^3
\]
Next, taking a union bound over all choices of $j_1<j_2<j_3$, 
$$
\Prob\big(\big\{i\in \calM(T,k)\big\}\cap 
\Event_{\text{\tiny\ref{lem:discub}}}(k-1)\big)
=
\Big(O\big(d^2\cdot (\log n)^{-2\cdot 2^{k-1}}\big)
\Big)^3.
$$
Finally, applying Markov's inequality, we obtain
$$
\Prob\big(\Event_{\text{\tiny\ref{lem:discub}}}(k)^c\cap
\Event_{\text{\tiny\ref{lem:discub}}}(k-1)\big)
\leq \big((\log n)^{2\cdot 2^{k}}\big)
\Big(O\big(d^2\cdot (\log n)^{-2\cdot 2^{k-1}}\big)
\Big)^3
=O\big(d^6\cdot (\log n)^{-2\cdot 2^{k-1}}\big),
$$
and thus
$$
\Prob\big(\Event_{\text{\tiny\ref{lem:discub}}}(k))
\geq \Prob(\Event_{\text{\tiny\ref{lem:discub}}}(k-1))
-\Prob\big(\Event_{\text{\tiny\ref{lem:discub}}}(k)^c\cap
\Event_{\text{\tiny\ref{lem:discub}}}(k-1)\big)
\geq \Prob(\Event_{\text{\tiny\ref{lem:discub}}}(k-1))
-(\log n)^{-1}.
$$
The induction is complete and the result follows. 
\end{proof}

As an immediate corollary of \cref{lem:discub}, we obtain
the desired performance guarantee for Algorithm~\ref{alg} that constitutes the upper bound in Theorem~\ref{thm:ub}.
\begin{corollary}
With probability at least $1
-\frac{O(1)}{\log n}-
\frac{2\log\log n}{\log n}$,
Algorithm~\ref{alg} returns a vector of signs $\sigma$
satisfying 
$$
\max\limits_{t\leq n}\Big\|\sum_{s\leq t}
\sigma_s \col_s(A)\Big\|_\infty\leq (C_\alg+7)\log\log n.
$$
\end{corollary}
In particular, it suffices to take $C_\ub = C_\alg + 7 = 35$. 
\begin{proof}
Apply \cref{lem:discub} with $k:=\lfloor 2\log\log n\rfloor$ and note:
\[
    \Event_{\text{\tiny\ref{lem:discub}}}(k) = \{|\calM(T,k)|
\le n\, (\log n)^{-2\cdot 2^k}<1\} = \{\calM(T,k) = \emptyset\}.
\]

\end{proof}

\bibliographystyle{plain}
\bibliography{references}
\end{document}